\documentclass[letterpaper]{article}

\title{A finite goal set in the plane which is not a Winner}
\author{Wesley Pegden\footnote{
Department of Mathematics, 
Rutgers University (New Brunswick), 
110 Frelinghuysen Rd.,
Piscataway, NJ 08854-8019. 
Email: pegden@math.rutgers.edu}
}

\date{September 20, 2007}

\usepackage{amsmath,amsthm,amssymb,amscd,latexsym,pstricks,fp,calc}

\usepackage{cite}
\newcommand{\hhh}{{\cal H}}
\newcommand{\fff}{{\cal F}}

\newcommand{\sbs}{\subset}
\newcommand{\abs}[1]{\lvert #1 \rvert}

\newcommand{\N}{\mathbb{N}}
\newcommand{\R}{\mathbb{R}}

\newtheorem{theorem}{Theorem}[section]
\newtheorem{conjecture}[theorem]{Conjecture}
\newtheorem{fact}[theorem]{Fact}
\newtheorem{lemma}[theorem]{Lemma}

{
\theoremstyle{definition}
\newtheorem{definition}[theorem]{Definition}

\newtheorem{q}{}
}

{
\theoremstyle{remark}

}

\SpecialCoor

\newcommand{\aaa}{27}

\begin{document}

      \psset{unit=2cm,dash=3pt 2pt,linewidth=.5pt,labelsep=.075,dotsize=4pt}

%%Let's compute the multiples of \aaa:
\FPmul{\aaat}{2}{\aaa}
\FPclip{\aaac}{\aaat}
\FPmul{\aaat}{3}{\aaa}
\FPclip{\aaad}{\aaat}

%%%compute the diaganol.:
\FPdiv{\halfaaat}{\aaa}{2}
\FPclip{\halfaaa}{\halfaaat}

\FPadd{\threehalvesaaat}{\aaa}{\halfaaa}
\FPclip{\threehalvesaaa}{\threehalvesaaat}

\FPadd{\fivehalvesaaat}{\aaa}{\threehalvesaaa}
\FPclip{\fivehalvesaaa}{\fivehalvesaaat}

\FPadd{\oppt}{180}{\threehalvesaaa}
\FPclip{\opp}{\oppt}

\FPadd{\ddt}{0}{\threehalvesaaa}
\FPclip{\dd}{\ddt}

\maketitle

\begin{abstract}
J.~Beck has shown that if two players alternately select previously unchosen points from the plane, Player 1 can always build a congruent copy of any given finite goal set $G$, in spite of Player 2's efforts to stop him \cite{ttt}.  We give a finite goal set $G$ (it has 5 points)  which Player 1 cannot construct \emph{before} Player 2 in this achievement game played in the plane.
\end{abstract}

\section{Introduction}
In the $G$-achievement game played in the plane, two players take turns choosing single points from the plane which have not already been chosen.  A player achieves a \emph{weak win} if he constructs a set congruent to the \emph{goal set} $G\sbs \R^2$ made up entirely of his own points, and achieves a \emph{strong win} if he constructs such a set before the other player does so.  (So a `win' in usual terms, \emph{e.g.}, in Tic-Tac-Toe, corresponds to a strong win in our terminology.)  This is a special case of a positional hypergraph game, where players take turns choosing unchosen points (vertices of the hypergraph) in the hopes of occupying a whole edge of the hypergraph with just their own points.  \cite{Bf,ttt} contain results and background in this more general area.

The type of game we are considering here is the game-theoretic cousin of Euclidean Ramsey Theory (see \cite{euc} for a survey).  Fixing some $r\in \N$ and some  finite point set $G\sbs \R^2$, the most basic type of question in Euclidean Ramsey Theory is to determine whether it is true that in every  $r$-coloring of the plane, there is some monochromatic congruent copy of $G$.

Restricting ourselves to 2 colors, the game-theoretic analog asks when Player 1 has a `win' in the achievement game with $G$ as a goal set.  Though one can allow transfinite move numbers indexed by ordinals (see Question \ref{trans} in Section \ref{qs}), it is natural to restrict our attention to games of length $\omega$, in which moves are indexed by the natural numbers.  In this case, a weak or strong winning strategy for a player is always a finite strategy (\emph{i.e.}, must always result in weak or strong win, respectively, in a finite, though possibly unbounded, number of moves) so long as the goal set $G$ is finite.  J. Beck has shown \cite{ttt}  that both players have strategies which guarantee them a weak win in finitely many moves for \emph{any} finite goal set---the proof is a potential function argument related to the classical Erd\H{o}s-Selfridge theorem\cite{es}.   The question of when the first player has a strong win---that is, whether he can construct a copy of $G$ first---seems in general to be a \emph{much} harder problem.  (A strategy stealing argument shows that the second player cannot have a strategy which ensures him a strong win: see Lemma \ref{steal}.)

For some simple goal sets, it is easy to give a finite strong winning strategy for Player 1.  This is the case for any goal set with at most 3 points, for example, or for the 4-point vertex-set of any parallelogram.  We give a set $G$ of 5 points for which we prove that the first player cannot have a finite strong win in the $G$-achievement game (proving, for example, that such finite goal sets do in fact exist).  This answers a question of Beck (oral communication).

Fix $\theta=t\pi$, where $t$ is irrational and $t<\frac{1}{9}$. Our set $G$ is a set of 5 points $g_i$, $1\leq i\leq 5$, all lying on a unit circle $C$ with center $c\in \R^2$.  For $1\leq i\leq 3$, the angle from $g_i$ to $g_{i+1}$ is $\theta$.  The point $g_5$ (the `middle point') is the point on $C$ lying on the bisector of the angle $\angle g_2 c g_3$.  (See Figure \ref{G}.)  We call this set the \emph{irrational pentagon}.

\begin{figure}
  \begin{center}
    \begin{pspicture}(-1,-1)(1,1)

	\psline[linestyle=dotted,linewidth=1pt](1;\dd)(0,0)

	\uput*[\aaa](1;0){\small$g_1$}
	\uput*{.3}[\halfaaa](0,0){\small$\theta$}

	\uput*[\aaa](1;\aaa){\small$g_2$}
	\uput*{.3}[\threehalvesaaa](0,0){\small$\theta$}

	\uput*{.3}[\fivehalvesaaa](0,0){\small$\theta$}
	\uput*[\aaa](1;\aaac){\small$g_3$}

	\uput*[\aaad](1;\aaad){\small$g_4$}

	\uput*[\dd](1;\dd){\small$g_5$}

	\uput*[-45](0,0){\small$c$}

        \pscircle[linewidth=.5pt,linestyle=dashed,dash=3pt 2pt](0,0){1}

	\psdot(1;0)  % semicolon indicates polar coordinates
	\psline[linestyle=dashed](0,0)(1;0)

	\psdot(1;\aaa)
	\psline[linestyle=dashed](0,0)(1;\aaa)

	\psdot(1;\aaac)
	\psline[linestyle=dashed](0,0)(1;\aaac)

	\psdot(1;\aaad)
	\psline[linestyle=dashed](0,0)(1;\aaad)

	\psdot(1;\dd)

    \end{pspicture}
  \end{center}
\label{G}
\caption[The goal set $G$]{The goal set $G$.  Player 2 can force a draw when the goal is this set, where $\theta$ is any irrational multiple of $\pi$ less than $\frac{\pi}{9}$.}
\end{figure}
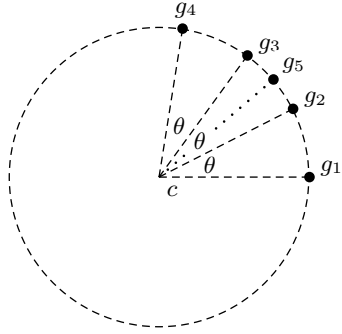

\begin{theorem}
There is no finite strong winning strategy for Player 1 in the $G$-achievement game when $G$ is the irrational pentagon.
\label{nfswin}
\end{theorem}

\noindent\textbf{Idea:} Let $\theta^n_c(x)$ denote the image of $x\in \R^2$ under the rotation $n\theta$ about the point $c$.  An important property of the irrational pentagon is that once a player has threatened to build a copy of it by selecting all the points $g_1,g_2,g_3,g_4$, he can give a new threat by choosing the point $\theta_c(g_4)$  or $\theta_c^{-1}(g_1)$. Furthermore, since $\theta$ is an irrational multiple of $\pi$, the player can continue to do this indefinitely, tying up his opponent (who must continuously block the new threats by selecting the corresponding middle points) while failing himself to construct a copy of $G$.  If Player 1 is playing for a finite strong win, he cannot let Player 2 indefinitely force in this manner.  However, to deny Player 2 that possibility, we will see that Player 1's only option is the same indefinite forcing, which leaves him no better.  The rest of the rigorous proof is a case study.

\section{The Proof}
For the proof of Theorem \ref{nfswin}, we will need the following lemma.

\begin{lemma} 
There are no three unit circles $C_1,C_2,C_3$ so that the pairs $C_i$,$C_j$ each intersect at 2 distinct points $x_{ij}$ and $y_{ij}$, so that the angles $\angle x_{ij}c_iy_{ik}$ are less than $\frac{\pi}{3}$ for all $j\neq i\neq k$.\label{noclose}
\end{lemma}

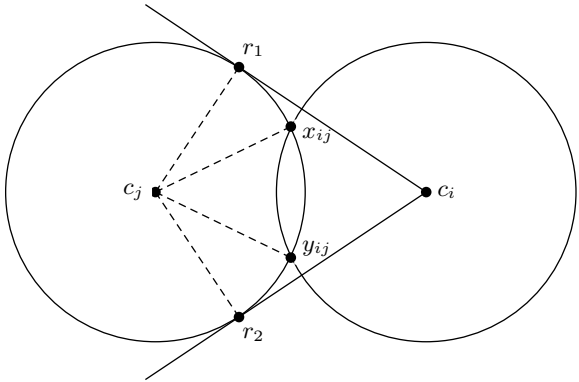
\begin{figure}
  \begin{center}
    \begin{pspicture}(-1,-1)(3,1)

        \pscircle(1.8,0){1}
	\uput*[0](1.8,0){\small$c_i$}

	\psdot(1.8,0)

	\psdot(0,0)

	\uput*[-25](1;25.8419){\small$x_{ij}$}
	\uput*[25](1;-25.8419){\small$y_{ij}$}
	\psline[linestyle=dashed](0,0)(1;25.8419)
	\psline[linestyle=dashed](0,0)(1;-25.8419)
	\psdot(1;25.8419)
	\psdot(1;-25.8419)

	\uput*[-56.24101](1;-56.25101){\small$r_2$}
	\psline(1.8,0)(-.0666666,1.2472191)
	\psdot(1;-56.25101)
	\psline[linestyle=dashed](0,0)(1;-56.25101)

	\uput*[56.24101](1;56.25101){\small$r_1$}
	\psline(1.8,0)(-.0666666,-1.2472191)
	\psdot(1;56.25101)
	\psline[linestyle=dashed](0,0)(1;56.25101)

	\uput*[180](0,0){\small$c_j$}
        \pscircle(0,0){1}

    \end{pspicture}
  \end{center}
\label{circs}
\caption{Proving Lemma \ref{noclose}}
\end{figure}

\begin{proof}
Let $B_i$ denote the unit ball whose boundary is $C_i$ for each $i$, and choose $C_i$ and $C_j$ from $\{C_1,C_2,C_3\}$ so that the area $A(B_i\cap B_j)$ is maximal.  In Figure 2, for any $C_k$ intersecting the circle $C_j$ at points $x_{jk},y_{jk}$ lying on $C_j$ between $r_1$ and $r_2$, we would have $A(B_i\cap B_k)>A(B_i\cap B_j)$, a contradiction.  The maximum angle between the points $r_1$ and $r_2$ on $C_j$ is $\frac{2\pi}{3}$.
\end{proof}

We are now ready to prove Theorem \ref{nfswin}.  Let $G$ now denote the irrational pentagon.

It is clear that Player 2 can either play indefinitely or reach a point where it is his move, he has a point $h_1$ at least 10 units away from any of Player 1's points, and Player 1 has no more than 2 points in any given (closed) ball of radius 10. (For example: on each turn until he has reached this point, Player 2 moves at least 30 units away from all of Player 1's points.)  Reaching this point, Player 2 begins to build a copy of $G$; that is, he arbitrarily designates some `center point' $c$ at unit distance from the point $h_1$, and chooses as his move a point $h_2$ which is an angle $\theta$ away from $h_1$ on the unit circle $C$ centered at $c$.  In fact, $h_1$ and $h_2$ lie on two unit circles which are disjoint except at $h_1,h_2$, and so Player 1's response can lie on only one of them; thus we assume without loss of generality that his response does not lie on the circle $C$.

Following Player 1's response, Player 2 will continue constructing his threat by choosing the point $h_3$ which lies on the circle $C$ and is separated from the points $h_1,h_2$ by angles $2\theta,\theta$, respectively.  Thus regardless of Player 1's choice of response, we see that Player 2 can reach the following situation:

\leavevmode\reversemarginpar\marginpar{\hfill$(\star)$}
It is Player 1's turn, Player 2 has points $h_1,h_2$,$h_3$ separated consecutively by angles $\theta$ on a unit circle $C$ centered at $c$, and Player 1 has at most $3$ points in any  unblocked copy of $G$.  Finally, Player 1 does not control 4 points of any unblocked copy of $G$, and  controls at most one point within 8 units of $c$. 

\leavevmode\reversemarginpar\marginpar{\hfill$(\star\star)$}
Moreover, there is in fact at most one unblocked copy of $G$ on which Player 1 has 3 points, and, if it exists, Player 1 controls no other points within (say) 5 units of those 3 points.

\vspace{1em}
We classify the rest of the proof into Cases 1,2,3 based on Player 1's move.  The analysis in Cases 1 and 3 depend just on the conditions in paragraph $(\star)$, while Case 2 depends on the conditions in both paragraphs $(\star)$ and $(\star\star)$.

\textbf{Case 1:}
A natural response for Player 1 might be to play on the circle $C$, thus attempting to prevent Player 2 from building a significant threat.  Since no point is a rotation of $h_1$ about the point $c$ by both positive and negative integer multiples of $\frac{\theta}{2}$, we may assume WLOG that Player 1 does not choose any rotations of $h_1$ about $c$ by positive integer multiples of $\frac{\theta}{2}$.  Thus Player 2 responds by choosing the point $h_4$ on $C$ which is at an angle $\theta,2\theta,3\theta$ from the points $h_3,h_2,h_1$, respectively.  Since Player 2 is now threatening to build a copy of $G$ on his next move and Player 1 is not (he has $\leq  3$ points on any unblocked copy of $G$), Player 1 must take the point on $C$ which together with $h_1,h_2,h_3,h_4$ complete a copy of $G$.  Player 2's response is naturally to choose the point $h_5$ on $C$ at angle $\theta,2\theta,3\theta$ from $h_4,h_3,h_2$, and we are in essentially the same situation: Player 1 has always at most 3 points in any unblocked congruent copy of $G$ (since he has only one point `near' $C$ which is not on $C$, and any set congruent to $G$ and not on $C$ intersects $C$ in at most 2 points), and Player 2 can force indefinitely.

\textbf{Case 2:}
Another response for Player 1 which may be possible is to play within the vicinity of his previously chosen points such that he controls 4 points of an unblocked copy of $G$.  By $(\star\star)$ Player 1 has only one 4-point threat, and so Player 2 can choose the corresponding fifth point to avoid losing.  Now, Player 1 may be able to continue to make threats on his subsequent moves, but it is easy to check using the conditions of $(\star\star)$ that his moves will have to stay on a single unit circle $C_1$ to do so, and that he will never be able to generate more than one threat, and thus never be able end his indefinite forcing with a win.  On the other hand, each time it is Player 1's move, the conditions in paragraph $(\star)$ are still satisfied, and so any move other than a continuation of the forcing will allow the analysis from Cases 1 and 3 to apply.

\textbf{Case 3:}  Finally, we consider the case where Player 1 does `none of the above'; that is, he chooses a point not on the circle $C$, but which nevertheless does not increase to 4 the number of points he controls in some congruent copy of $G$.  This is the case where we make use of Lemma \ref{noclose}.

  Player 1 now has as many as two points within 8 units distance of the point $h_1$.  By choosing successively points $h_4,h_5,h_6,$ etc., as in Case 1, Player 2 hopes to successively force Player 1 to take the corresponding fifth point of each congruent copy of $G$ that Player 2 threatens to build at each step.  The only snag is this: it is conceivable that Player 1, in taking these corresponding `fifth' points, builds his own threat.  He already has two points in the vicinity, and it is possible that they lie on a congruent copy of $G$ which intersects the circle $C$ in two points which Player 1 will eventually be forced to take by Player 2's moves.  In this case, Player 2 would have to respond and could conceivably end up losing the game if Player 1 is able to break is forcing sequence.

Of course, this is only truly a problem if Player 1 is threatening this in `both directions'---that is, regardless of whether $h_4$ is at angles $\theta,2\theta,3\theta$ to the points $h_3,h_2,h_1$, respectively, or to the points $h_1,h_2,h_3$, respectively.  However, such a double threat is immediately ruled out by Lemma \ref{noclose}, since this would require two sets $S_1,S_2\cong G$ (each a subset of a $3\theta$-arc of a unit circle) intersecting each other in two points (previously chosen by Player 1) and each also each intersecting $C$ in two places.  This completes the proof.\qed

\section{Further Questions}
\label{qs}
\begin{q}
Our (rather crude) methods do not appear suited to much larger goal sets.  So we ask: are there arbitrarily large goal sets $G$ for which Player 1 cannot force a finite strong win in the $G$-achievement game played in $\R^2$?
\end{q}

\begin{q}
We have examples of 4-point sets for which Player 1 has strong winning strategies, and we have given here a 5-point example where Player 2 has a drawing strategy.  Are there 4-point sets where Player 2 has a drawing strategy?
\end{q}

\begin{q}
Player 1 can easily be shown to have strong winning strategies for any goal set of size at most 3, and any 4-point goal set which consists of the vertices of a parallelogram.  It is not difficult to give a 5 point goal set for which Player 1 can be shown to have a strong winning strategy.  Are there arbitrarily large goal sets $G$ for which Player 1 has a strong winning strategy?
\end{q}

\begin{q}
We restricted our attention here to the first $\omega$ moves, and indeed, our proof does not show that Player 1 can't force a strong win if transfinite move numbers are allowed.  So we ask: are there finite sets $G$ for which Player 1 cannot force a strong win, when the players make a move for each successor ordinal?\label{trans}
\end{q}

\begin{q}
In the general achievement game played on a hypergraph (in which the two players select vertices, and the goal sets are the edges)
 we define some stronger win types for Player 1:
\psset{xunit=1.4cm,yunit=1.4cm}
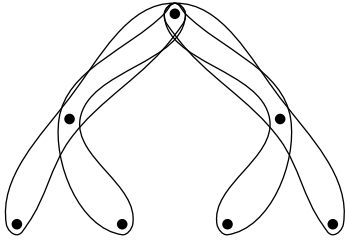
\begin{figure}
\begin{center}
\begin{pspicture}(0,0)(4,2)
\psdot*(2,2)

\psdot*(1,1)  \psdot*(3,1)

\psdot*(.5,0)  \psdot*(1.5,0) \psdot*(2.5,0) \psdot*(3.5,0)

\psccurve[showpoints=false](2,2.1)(1,1.2)(.4,0)(.5,-.1)(.6,0)(1,.8)(2,1.8)(2.1,2)(2,2.1)

\psecurve[showpoints=false](1.9,2)(2,2.1)(3,1.2)(3.6,0)(3.5,-.1)(3.4,0)(3,.8)(2,1.8)(1.9,2)(2,2.1)(3,1.2)

\psccurve[showpoints=false](2,2.1)(1.9,2)(2.9,1)(2.4,0)(2.5,-.1)(3.1,1)(2.1,2)(2,2.1)

\psecurve[showpoints=false](1.9,2)(2,2.1)(2.1,2)(1.1,1)(1.6,0)(1.5,-.1)(.9,1)(1.9,2)(2,2.1)(3,1.2)

\end{pspicture}
\end{center}
\caption{The hypergraph $\hhh_T$, in the case where $T$ is the balanced binary directed tree of depth 2.\label{HT}}
\end{figure}
\begin{definition}
In the achievement game played on a hypergraph $\hhh$, Player 1 has a \emph{fair win} if he builds some $e\in E(\hhh)$ on a turn which comes before any turn on which Player 2 builds some $f\in E(\hhh)$.\label{fair}
\end{definition}
Each `turn' of the game consists of a move by Player 1 followed by a move by Player 2.  Definition \ref{fair} requires simply that Player 1 builds a goal set in fewer turns than it takes Player 2 to do the same (if Player 2 can at all).
\begin{definition}
In the achievement game played on a hypergraph $\hhh$, Player 1 has an \emph{early win} if he builds some $e\in E(\hhh)$, say in $n$ moves, such that there is no $m\leq n$ for which Player 2 had $\abs{e}-1$ points of a set $e\in E(\hhh)$ on his $m$th turn, and on which Player 1 had no point on his $m$th turn.
\end{definition}

So every early win is a fair win, and every fair win is a strong win.  In general, none of the win types we have defined are the same, and they all occur for Player 1 for some hypergraph: Already for $K_4$, Player 1 has a strong win but not a fair win.  On the hypergraph $\hhh_T$, whose vertices are the vertices of some balanced binary directed tree $T$, and whose edges are the vertex-sets of longest directed paths in $T$ (Figure \ref{HT}), Player 1 has a fair win and an early win.  Finally, let the hypergraph $\fff_n$ have vertex set $[n]\times \{0,1\}$.  Edges are of two types: Type 1 edges are the $n$-subsets $S\sbs [n]\times\{0,1\}$ for which the $\pi_1(S)=[n]$ and $(1,0)\in S$, and Type 2 edges are all the pairs $\{(m,0),(m,1)\}$ where $m\in [n]$ (see Figure \ref{Fn}).  Player 1 has a fair win in $\fff_n$ for $n\geq 2$, but not an early win.  Probably, however, the situation is not so rich in the plane:

\begin{conjecture} There is no finite point set $G\sbs \R^2$ for which Player 1 has a strategy which ensures a fair win in the $G$-achievement game played in the plane.
\end{conjecture}
 The conjecture may seem painfully obvious.  If we play the achievement game in $\R\setminus\{c\}$ for any point $c\in \R^2$, for example, Player 2 can prevent a fair win by always choosing the point which is the central reflection across $c$ of Player 1's last move.  Annoyingly, even proving that Player 1 cannot have an early win for any $G$ when playing in $\R^2$ may be very difficult.

\psset{xunit=1.4cm,yunit=1.4cm}
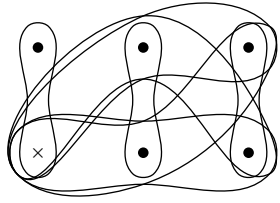
\begin{figure}
\begin{center}
\begin{pspicture}(.9,-.5)(3,1.5)

\psdot(3,0)\psdot(3,1)
\psccurve(2.9,-.2)(2.9,.5)(2.9,1.2)(3.1,1.2)(3.1,.5)(3.1,-.2)

\psdot(2,0)\psdot(2,1)
\psccurve(1.9,-.2)(1.9,.5)(1.9,1.2)(2.1,1.2)(2.1,.5)(2.1,-.2)

\psdot[dotstyle=x](1,0)\psdot(1,1)
\psccurve(0.9,-.2)(0.9,.5)(0.9,1.2)(1.1,1.2)(1.1,.5)(1.1,-.2)

\psccurve(.8,-.2)(2,-.3)(3.2,-.2)(3.2,.2)(2,.3)(.8,.2)
\psccurve(.8,-.2)(2,.7)(3.2,.8)(3.2,1.2)(2,1.3)(.8,-.2)
\psccurve(.8,-.2)(2,-.3)(3.2,.8)(3.2,1.2)(2,.3)(.8,.2)
\psccurve(.8,-.2)(2,.7)(3.2,-.2)(3.2,.2)(2,1.3)(.8,-.2)

\end{pspicture}
\end{center}
\caption{The hypergraph $\fff_3$.  There are four (in general $2^{n-1}$) Type 1 edges, and three (in general $n$) Type 2 edges. (The vertex $(1,0)$ is marked with $\times$.)\label{Fn}}
\end{figure}

For the sake of completeness, we note the situation on the hypergraph $\hhh_T$ is in some way the worst possible for Player 2.  It is easy to see that although Player 2 never occupies all but one vertex of an unblocked edge when playing on $\hhh_T$, it is easy for him to occupy all but one vertex of some edge which may be blocked.  The natural strengthening of the `early win' suggested here never occurs for Player 1:

\begin{definition}
In the achievement game played on a hypergraph $\hhh$, Player 1 has a \emph{humiliating win} if he occupies some $e\in E(\hhh)$ before Player 2 occupies all but one vertex of some edge $f\in E(\hhh)$.
\label{hum}
\end{definition}
\noindent(So every humiliating win is an early win.)  The fact that Player 1 never has a humiliating win will follow from the strategy stealing argument; we include the proof for completeness.
\begin{lemma}[Strategy Stealing]
On any hypergraph $\hhh$, a second player cannot have a strategy which ensures strong win in the achievement game.
\label{steal}
\end{lemma}
\begin{proof}
The proof of Lemma \ref{steal} is the strategy stealing argument; we include the proof for completeness.   We argue by contradiction:   if the second player has a strong win strategy $\sigma$ (a function from game positions to vertices), the first player makes an arbitrary first move $g$ (his ghost move).  Now on each move, the first player mimics the second player's strategy by ignoring his ghost move: formally, let $G_n$ denote the game's position on the $n$th move, and let $G_n\setminus x$ denote the game position modified so that the vertex $x$ is unchosen.  Then on each turn, the first player chooses the point $\sigma(G_n\setminus g)$ if it is not equal to $g$ (and thus must be unoccupied, since $\sigma$ is a valid strategy), or, if $\sigma(G_n\setminus g)=g$, the first player chooses an arbitrary point $x\in V(\hhh)$ and sets $g:=x$.  The fact that $\sigma$ was a `strong win' strategy for the second player implies that the first player will occupy all of an edge $e\in E(\hhh)$ (even requiring $e\not\ni g$) before the second player occupies all some some edge $f\in E(\hhh)$.  In particular, the first player has a strong win, a contradiction.
\end{proof}
\begin{fact}
On any hypergraph $\hhh$, Player 2 can prevent Player 1 from achieving a humiliating win.
\label{nohum}
\end{fact}
\begin{proof}
Denote by $x$ the vertex Player 1 chooses on his first move.  The hypergraph $\hhh\setminus x$ is the hypergraph with vertex-set $V\setminus\{x\}$ and edges $e\setminus \{x\}$ for each $e\in E(H)$.  We see that Player 1 has a humiliating win on $\hhh$ only if he has a strong win on $\hhh\setminus \{x\}$ \emph{as a second player}, and we are done by Lemma \ref{steal}.  
\end{proof}

Lemma \ref{steal} is deceptive in its simplicity.  Of course we emphasize that the strategy stealing argument shows only the existence of a strategy for a first player to prevent a second player strong win.  In general, we have no better way to find such a strategy than the na\'ive `backwards labeling' method, which runs on the whole game tree.  Thus, though Fact \ref{nohum} tells us that Player 2 should never fall more than one behind Player 1 (in the sense of Definition \ref{hum}), it is quite possible for this to happen in actual play between good (yet imperfect) players.
\end{q}

\subsubsection*{Acknowledgment}
I'd like to thank J\'ozsef Beck for discussing with me the questions I consider here, and for helpful suggestions regarding the presentation.

\end{document}